\documentclass[a4paper,12pt]{amsart}
\usepackage{amssymb}
\usepackage{ifthen}
\usepackage{cite}
 \usepackage[dvips]{graphicx}
\nonstopmode \numberwithin{equation}{section}
\usepackage{amssymb}
\usepackage{ifthen}
\usepackage{graphicx}
\usepackage{amsmath}
\usepackage[T1]{fontenc} %skandit
\usepackage[utf8]{inputenc}
\usepackage[usenames,dvipsnames]{color}
\usepackage{color}
\usepackage[english]{babel}
\usepackage{fancyhdr}
\usepackage{fancybox}
\usepackage{tikz}
\nonstopmode \numberwithin{equation}{section}
\theoremstyle{plain}

\newtheorem{conj}{Conjecture}

\theoremstyle{definition}
\newtheorem{defn}{Definition}[section]

\newtheorem{thm}{Theorem}[section]
\newtheorem{thm-A}{Theorem A}[section]
\newtheorem{prob}{Problem}[section]
\newtheorem{cor}{Corollary}[section]
\newtheorem{ques}{Question}[section]
\newtheorem{prop}{Proposition}[section]
\newtheorem{rem}{Remark}[section]
\newtheorem{lem}{Lemma}[section]
\newtheorem{defi}{Definition}[section]

\newcounter{minutes}
\divide\time by 60
\newcounter{hours}
\multiply\time by 60
\addtocounter{minutes}{-\time}
\newcounter {own}
\def\theown {\thesection       .\arabic{own}}

\newenvironment{pf}[1][]{%
 \vskip 3mm
 \noindent
 \ifthenelse{\equal{#1}{}}%
  {{\slshape Proof. }}%
  {{\slshape #1.} }%
 }%
{\qed\bigskip}

\newcounter{alphabet}

\newcommand{\real}{{\operatorname{Re}\,}}

\def\be{\begin{equation}}
\def\ee{\end{equation}}

\newcommand{\bee}{\begin{enumerate}}
\newcommand{\eee}{\end{enumerate}}

\newcommand{\blem}{\begin{lem}}
\newcommand{\elem}{\end{lem}}
\newcommand{\bthm}{\begin{thm}}
\newcommand{\ethm}{\end{thm}}
\newcommand{\bcor}{\begin{cor}}
\newcommand{\ecor}{\end{cor}}
\newcommand{\beg}{\begin{examp}}
\newcommand{\eeg}{\end{examp}}
\newcommand{\begs}{\begin{examples}}
\newcommand{\eegs}{\end{examples}}

\newcommand{\bdefn}{\begin{defn}}
\newcommand{\edefn}{\end{defn}}

\newcommand{\bprob}{\begin{prob}}
\newcommand{\eprob}{\end{prob}}
\newcommand{\bei}{\begin{itemize}}
\newcommand{\eei}{\end{itemize}}

\newcommand{\bcon}{\begin{conj}}
\newcommand{\econ}{\end{conj}}
\newcommand{\bcons}{\begin{conjs}}
\newcommand{\econs}{\end{conjs}}
\newcommand{\bprop}{\begin{prop}}
\newcommand{\eprop}{\end{prop}}
\newcommand{\br}{\begin{rem}}
\newcommand{\er}{\end{rem}}
\newcommand{\brs}{\begin{rems}}
\newcommand{\ers}{\end{rems}}
\newcommand{\bo}{\begin{obser}}
\newcommand{\eo}{\end{obser}}
\newcommand{\bos}{\begin{obsers}}
\newcommand{\eos}{\end{obsers}}
\newcommand{\bpf}{\begin{pf}}
\newcommand{\epf}{\end{pf}}
\newcommand{\ba}{\begin{array}}
\newcommand{\ea}{\end{array}}
\newcommand{\beq}{\begin{eqnarray}}
\newcommand{\beqq}{\begin{eqnarray*}}
\newcommand{\eeq}{\end{eqnarray}}
\newcommand{\eeqq}{\end{eqnarray*}}

\begin{document}

\title{Revisiting Bohr Inequalities with Analytic and Harmonic Mappings on unit disk}

\author{Molla Basir Ahamed}
\address{Molla Basir Ahamed, Department of Mathematics, Jadavpur University, Kolkata-700032, West Bengal,India.}
\email{mbahamed.math@jadavpuruniversity.in}

\author{Partha Pratim Roy}
\address{Partha Pratim Roy, Department of Mathematics, Jadavpur University, Kolkata-700032, West Bengal,India.}
\email{pproy.math.rs@jadavpuruniversity.in}

\subjclass[{AMS} Subject Classification:]{Primary 30A10, 30H05, 30C35, 30C50 Secondary 30C45}

\def\thefootnote{}
\footnotetext{ {\tiny File:~\jobname.tex,
printed: \number\year-\number\month-\number\day,
          \thehours.\ifnum\theminutes<10{0}\fi\theminutes }
} \makeatletter\def\thefootnote{\@arabic\c@footnote}\makeatother
\begin{abstract} 
In this paper, we study some improved and refined versions of the classical Bohr inequality applicable to the class $\mathcal{B}$, which consists of self-analytic mappings defined on the unit disk $\mathbb{D}$. First, we improve the Bohr inequality for the class $\mathcal{B}$ of analytic self-maps, incorporating the area measurements of sub-disks $\mathbb{D}_r$ of $\mathbb{D}$. Secondly, we establish a sharp inequality with suitable setting as an improved version of the classic Bohr inequality. Then we obtain a sharp refined Bohr inequality in which the coefficients $|a_k|$ $(k=0, 1, 2, 3)$ in the majorant series $M_f(r)$ of $f$ are replaced by $|f^{(k)}(z)|/k!$. Finally, for a certain class $\mathcal{P}^0_{\mathcal{H}}(M)$ of harmonic mappings of the form $f=h+\overline{g}$, we generalize the Bohr inequality incorporating a sequence $\{\varphi_n(r)\}_{n=0}^{\infty}$ of continuous functions of $r$ in $[0, 1)$ and give some applications.
\end{abstract}

\maketitle
\pagestyle{myheadings}
\markboth{M. B. Ahamed and P. P. Roy}{Revisiting Bohr Inequalities with Analytic and Harmonic Mappings on unit disk}
\section{Introduction}
Let $ \mathbb{D}:=\{z\in \mathbb{C}:|z|<1\} $ denote the open unit disk in $ \mathbb{C} $. Herald Bohr in $1914$ (see \cite{Bohr-1914}) states that if $H_\infty$ denotes the class of all bounded analytic functions $f$ on $\mathbb{D}$, then the following inequality holds
\begin{align}\label{Eq-1.1}
	B_0(f,r):=|a_0|+\sum_{n=1}^{\infty}|a_n|r^n\leq||f||_\infty:=\displaystyle\sup_{z\in\mathbb{D}}|f(z)| \;\;\mbox{for}\;\; 0\leq r\leq\frac{1}{3},
\end{align}
where $a_k=f^{(k)}(0)/k!$ for $k\geq 0$. The constant $1/3$ is called the Bohr radius and the inequality in \eqref{Eq-1.1} is called Bohr inequality for the class $\mathcal{B}$ of analytic self-maps on $\mathbb{D}$. Henceforth, if there exists a positive real number $r_0$ such that an inequality of the form \eqref{Eq-1.1} holds for every elements of a class $\mathcal{M}$ for $0\leq r\leq r_0$ and fails when $r>r_0$, then we shall say that $r_0$ is an sharp bound for $r$ in the inequality w.r.t. to class $\mathcal{M}$.\vspace{1.2mm}

Initially, Bohr showed the inequality \eqref{Eq-1.1} holds for $|z|\leq1/6$,  and later  M. Riesz, I. Shur and F. Wiener, independently proved its validity on a wider range $0\leq r\leq 1/3$ and the number $1/3$ is optimal as seen by analyzing suitable members of the conformal automorphism of the unit disk $\mathbb{D}$.\vspace{1.5mm} %There are many proofs of this inequality (cf. \cite{Sidon-1927} and \cite{Tomic-1962}).\vspace{1.5mm}

 This result has created enormous interest on Bohr's inequality in various settings (see \cite{Bhowmik-2018,Boas-1997,Bombieri-2004,Chen-Liu-Pon-RM-2023,Defant-2011,Kayumov-CMFT-2017,Kayumov-2018-JMAA}. Initial exploration of this problem was conducted by Harald Bohr as he delved into the absolute convergence of Dirichlet series of the form $\sum a_n n^{-s}$. However, in recent years, it has evolved into a thriving domain of research in modern function theory. The theorem has drawn considerable interest because it plays a crucial role in solving the characterization problem for Banach algebras that satisfy the von Neumann inequality (see \cite{Dixon & BLMS & 1995}). However, this compelling inequality has been substantiated by several distinct proofs presented in various articles (see \cite{Sidon-1927, Paulsen-2002, Tomic-1962}). Boas and Khavinson (see \cite{Boas-1997}) were the first to introduce the concept of the n-dimensional Bohr radius ${\mathcal{K}}^{\infty}_n$ which expedites the research what is called the multidimensional Bohr radii.\vspace{1.2mm}

 In \cite{Ponnusamy-RM-2020,Ponnusamy-HJM-2021, Ponnusamy-JMAA-2022,Chen-Liu-Pon-RM-2023,Kumar-CVEE-2023}, the authors have delved into the Bohr radius in the context of a more general class of power series. In this investigation, they have established generalized Bohr inequalities in view of changing the basis of sequence  $ \{r^n\}_{n=0}^{\infty} $ with a sequence  $ \{\varphi_n(r)\}_{n=0}^{\infty} $, where $ \varphi_n(r) $ is a non-negative continuous function defined on the interval $ [0,1) $. It is noteworthy that the requirement for this new sequence is that $ \sum_{n=0}^{\infty}\varphi_n(r) $ should exhibit local uniform convergence with respect to the parameter $ r $ in the interval $ [0,1) $. Recent investigations led by  Ahamed \emph{et al.} \cite{Ahamed-AASFM-2022} and Evdoridis \emph{et al.} \cite{Ponnusamy-RM-2021} have been dedicated to exploring the Bohr Phenomenon for the class of analytic functions defined on shifted disk $ \Omega_{\gamma} $ ($ 0\leq \gamma<1 $) containing the unit disk $ \mathbb{D} $. Furthermore, the research in \cite{Allu-JMAA-2021} have directed towards analyzing the Bohr radius for certain classes of starlike and convex univalent functions in the unit disk $\mathbb{D}$. \vspace{1.2mm}
 
 Bohr’s inequality for holomorphic and pluriharmonic mappings with values in complex Hilbert spaces are obtained in \cite{hamada-MN-2023}. However, it can be noted that not every class of functions has Bohr phenomenon, for example, B\'en\'eteau \emph{et al.} \cite{Beneteau-2004} showed that there is no Bohr phenomenon in the Hardy space $ H^p(\mathbb{D},X), $ where $p\in [1,\infty).$ In \cite{Liu-Liu-JMAA-2020}, Liu and Liu have shown that Bohr's inequality fails to hold for the class $ \mathcal{H}(\mathbb{D}^2, \mathbb{D}^2) $ of holomorphic functions $ f : \mathbb{D}^2\rightarrow \mathbb{D}^2 $ having lacunary series expansion. Over the past few years, Bohr's inequality has gained a significant attention, resulting in extensive research and extensions in various directions with different settings (see \cite{Allu-CMFT,Boas-2000,Bhowmik-CMFT-2019,Boas-1997,Paulsen-PLMS-2002,Paulsen-BLMS-2006,Ponnusamy-CMFT-2020,Tomic-1962,Lata-Singh-PAMS-2022,Liu-JMAA-2021} and references therein). \vspace{1.2mm}
 
 In this paper, our aim is two fold. First, we establish some improved and refined Bohr inequalities for the class $\mathcal{B}$, and second, we obtain generalized Bohr inequality  for certain class of harmonic mappings in terms of a sequence of continuous functions. The organization of the paper is as follows: In Section 2, we study Bohr phenomenon for the class $\mathcal{B}$ of bounded analytic functions on unit disk $ \mathbb{D} $. In Section 3, generalization of the Bohr inequality for certain classes of harmonic mappings and its applications are discussed. The proof of the main results are discussed in each section individually.
\section{Bohr inequalities for the class $\mathcal{B}$ of bounded analytic functions on unit disk $ \mathbb{D} $} 
Before we continue the discussion, we fix some notations.
 \subsection{Basic Notations}  Let 
 \begin{align*}
 	\mathcal{B}=\{f\in H_{\infty} \;:\;||f||_{\infty}\leq 1\}.
 \end{align*} Also, for $f(z)=\sum_{n=0}^{\infty}|a_n|z^n\in \mathcal{B}$ and $f_0(z):=f(z)-f(0)$, we let for convenience 
 \begin{align*}
 	B_k(f,r):=\sum_{n=k}^{\infty}|a_n|r^n \;\mbox{for }k\geq 0,\; ||f_0||^2_r:=\sum_{n=1}^{\infty}|a_n|^2r^{2n} \;\mbox{and} \;||f_1||^2_r:=\sum_{n=2}^{\infty}|a_n|^2r^{2n}
 \end{align*}
so that
\begin{align*}
\begin{cases}
	 B_0(f,r)=|a_0|+B_1(f,r),\\   B_0(f,r)=|a_0|+|a_1|r+B_2(f,r),\\ B_0(f,r)=|a_0|+|a_1|r+|a_2|r^2+B_3(f,r), \; \mbox{so on}.
\end{cases}
 \end{align*}
 Further, we define the quantity $A(f_0,r)$ and  $A(f_1,r)$ by 
 \begin{align*}
 	\begin{cases}
 		A(f_0,r):=\left(\dfrac{1}{1+|a_0|}+\dfrac{r}{1-r}\right)||f_0||^2_r,\vspace{2mm}\\
 		A(f_1,r):=\left(\dfrac{1}{1+|a_0|}+\dfrac{r}{1-r}\right)||f_1||^2_r.
 	\end{cases}
 \end{align*}
 Kayumov \emph{et al.} \cite{Ponnusamy-2017,Kayu-Kham-Ponnu-2021-JMAA} have derived the Bohr-Rogosinski inequality for class $ \mathcal{B} $ based on the investigation of Rogosinski's inequality and Rogosinski's radius as presented in \cite{Rogosinski-1923}.
\begin{thm}\cite{Kayu-Kham-Ponnu-2021-JMAA}\label{th-1.9}
	 Suppose that $ f(z)=\sum_{n=0}^{\infty}a_nz^n\in\mathcal{B} $. Then 
\begin{align*}
	 |f(z)|+B_N(f,r)\leq 1\;\;\mbox{for}\;\; r\leq R_N,
\end{align*}
where $ R_N $ is the positive root of the equation $ 2(1+r)r^N-(1-r)^2=0 $. The radius $ R_N $ is the best possible. Moreover, 
\begin{align*}
		|f(z)|^2+B_N(f,r)\leq 1\;\;\mbox{for}\;\; r\leq R^{\prime}_N,
\end{align*}
where $ R^{\prime}_N $ is the positive root of the equation $ (1+r)r^N-(1-r)^2=0 $. The radius $ R^{\prime}_N $ is the best possible.
\end{thm}
 This result generates a significant amount of research activity on Bohr-Rogosinki inequalities for different classes of functions. However, Liu \emph{et al.} \cite{Liu-Shang-Xu-JIA-2018} have proved the following result for functions in the class $ \mathcal{B} $, where $ |a_0| $ and $ |a_1| $ are replaced by $ |f(z)| $ and $ |f^{\prime}(z)| $, respectively.
\begin{thm}\cite{Liu-Shang-Xu-JIA-2018} \label{th-1.6}
Suppose that $ f(z)=\sum_{n=0}^{\infty}a_nz^n\in\mathcal{B} $. Then 
\begin{align*}
|f(z)|+|f^{\prime}(z)|r+B_2(f,r)\leq 1\;\; \mbox{for}\;\; r\leq(\sqrt{17}-3)/4.
\end{align*}
The constant $ (\sqrt{17}-3)/4 $ is best possible.
\end{thm}
Further,  several refined Bohr inequality with suitable setting are investigated in \cite{Liu-Liu-Ponnusamy-2021} and obtained the following result.
 \begin{thm}\label{Th-2.3} \cite{Liu-Liu-Ponnusamy-2021}
	Suppose that $ f(z)\in\mathcal{B} $ and $f(z)=\sum_{n=0}^{\infty}a_nz^n$. Then
	\begin{align}\label{Eq-2.1}
	|f(z)|+B_1(f,r)+A(f_0,r)\leq 1 
	\end{align}
	for $ |z|= r_0 \leq(2)/(3+|a_0|+\sqrt{5}(1+|a_0|))$ and the constant $r_0$ is the best possible and $r_0>\sqrt{5}-2$. Moreover, 
	\begin{align}\label{Eq-2.2}
		|f(z)|^2+B_1(f,r)+A(f_0,r)\leq 1,   
	\end{align} 
	for $|z|= r \leq r^{\prime}_0,$ where $r^{\prime}_0$ is the unique root in $(0, 1)$ of the equation
	\begin{align*}
		(1-|a_0|^3)r^3-(1+2|a_0|)r^2-2r+1=0
	\end{align*}
	and $r^{\prime}_0$ is the best possible and $1/3<r^{\prime}_0<1/(2+|a_0|)$. 
\end{thm}
Moreover, Liu \emph{et. al} (see \cite{Liu-Liu-Ponnusamy-2021}) refined the inequalities (\ref{Eq-2.1}) and (\ref{Eq-2.2}) by replacing the second coefficient $|a_1|$ in the majorant series $ B_0(f, r)$ by the quantity $|f^{\prime}(z)|$ and obtained the following result.
\begin{thm}\label{Thm-2.4} \cite{Liu-Liu-Ponnusamy-2021}
	Suppose that $ f(z)\in\mathcal{B} $ and $f(z)=\sum_{n=0}^{\infty}a_nz^n$. Then
	\begin{align}\label{Eqnnn-2.3}
		|f(z)|+|f^{\prime}(z)|r+B_2(f,r)+A(f_0,r)\leq 1 
	\end{align}
	for $ |z|= r \leq(\sqrt{17}-3)/4$ and the constant $(\sqrt{17}-3)/4$ is the best possible. Moreover, 
	\begin{align}\label{Eqnnn-2.4}
	&|f(z)|^2+|f^{\prime}(z)|r+B_2(f,r)+A(f_0,r)\leq 1 
	\end{align}
for $ |z|= r \leq r_0 $, where $ r_0\approx 0.385795 $ is the unique positive root of the equation $ 1-2r-r^2-r^3-r^4=0 $	and the constant $r_0$ is the best possible. 
\end{thm} 
After that,  the inequality (\ref{Eqnnn-2.3}) is refined (see \cite{Ahamed-CMFT-2022}) by   replacing the second coefficient $|a_2|$ by $|f^{\prime\prime}(z)|/2!$ and the following result is obtained.
\begin{thm}[\cite{Ahamed-CMFT-2022}]\label{Th-2.5}
	Suppose that $ f(z)=\sum_{n=0}^{\infty}a_nz^n\in\mathcal{B} $. Then 
	\begin{align}\label{Eq-2.5}
	|f(z)|+|f^\prime(z)|r+\frac{|f^{\prime\prime}(z)|}{2!}r^2+B_3(f,r)+\frac{|a_1|^2r^3}{1-r}+A(f_1,r)\leq 1
	\end{align}
	for $|z|=r\leq r_0\approx 0.287459$, where $r_0$ is the unique root in $(0,1)$ of the equation 
	\begin{align*}
			-1+3r+r^2+r^3+4r^4+2r^5=0
	\end{align*} 
	and the constant $r_0$ is best possible.
\end{thm} 
\subsection{Sharp refined Bohr inequality for the class $\mathcal{B}$}
An extensive research work on Bohr phenomenon are carried out by several authors for the class $ \mathcal{B} $ of analytic self maps on the unit disk $ \mathbb{D} $. Recently, the article \cite{Ahamed-CMFT-2022} has presented a sharp refined form of the inequality (\ref{Eqnnn-2.3}) mentioned in the article (see Theorem \ref{Thm-2.4}), which was originally proposed by Liu \emph{et al.} (refer to \cite{Liu-Liu-Ponnusamy-2021}). However, \cite{Ahamed-CMFT-2022} omits any elaboration on the significant refinement of the inequality \eqref{Eqnnn-2.4}. In light of this, to continue the research, a natural query might emerge, and investigating its response could be pertinent in this topic of research.
 \begin{ques}\label{Q-3.1}
Can we establish a sharp refinement of the inequality (\ref{Eqnnn-2.4}) by replacing the third coefficient $|a_2|$ by $|f^{\prime\prime}|/2!?$
\end{ques}
\begin{ques}\label{Q-3.2}
	Can we  the improve the classical Bohr inequality replacing the initial coefficients  $ |a_0|, |a_1| , |a_2|$ and $|a_3|$  in the majorant series by $ |f(z)|, |f^{\prime}(z)|, |f^{\prime\prime}(z)|/2! $ and $|f^{\prime\prime\prime}(z)|/3!$ respectively?
\end{ques}
In this section, we were inspired by their result to adopt the  proof technique of Theorem \ref{Th-2.5}, aiming to provide answers to the Question \ref{Q-3.1} and \ref{Q-3.2}. Consequently, we have obtained the following two results, to address the Questions \ref{Q-3.1} and \ref{Q-3.2} successfully.
 \begin{thm}\label{Thmm-2.3}
Suppose that $ f(z)=\sum_{n=0}^{\infty}a_nz^n\in\mathcal{B} $. Then we have  
\begin{align*}
\mathcal{D}_{f}(z,r):=|f(z)|^2+|f^\prime(z)|r+\frac{|f^{\prime\prime}(z)|}{2!}r^2+B_3(f,r)+\frac{|a_1|^2r^3}{1-r}+A(f_1, r)\leq 1
\end{align*}
for $|z|=r\leq r_0\approx 0.393727$, where $r_0$ is the unique root in $(0,1)$ of the equation $ -1+2r+r^2+2r^4+r^5=0  $ and the constant $r_0 $ is best possible.
\end{thm} 
To present our next result concisely, we need to define a functional. Henceforth, for $ f\in\mathcal{B} $ and $ p=1, 2 $, we define 
\begin{align*}
	\mathcal{J}_{f, p}(z,r):&=|f(z)|^p+|f^\prime(z)|r+\frac{|f^{\prime\prime}(z)|}{2!}r^2+\frac{|f^{\prime\prime\prime}(z)|}{3!}r^3+B_4(f,r)\\&\quad+|a_1|^2\frac{r^4}{1-r}+A(f_1, r).
\end{align*}
\begin{thm}\label{th-2.4}
	Suppose that $ f(z)=\sum_{n=0}^{\infty}a_nz^n\in\mathcal{B} $. Then we have 
	\begin{enumerate}
		\item[(i)] $ \mathcal{J}_{f, 1}(z,r)\leq 1$
		for $|z|=r\leq R_1\approx 0.285086$, where $r_0$ is the unique root of the equation $ -1+4r-2r^2+3r^4+2r^5-2r^6-2r^7=0  $	in $(0,1)$. The constant $R_1 $ is best possible. The equality $ \mathcal{J}_{f, 1}(z,r)= 1 $ is achieved for the function 
		$f_a$.
		\item[(ii)] $ \mathcal{J}_{f, 2}(z,r)\leq 1 $
		for $|z|=r\leq R_2\approx 0.386055$, where $r_0$ is the unique root of the equation $ -1+3r-r^2-r^3+2r^4+r^5-r^6-r^7=0  $ in $(0,1)$. The constant $R_2 $ is best possible. The equality $ \mathcal{J}_{f, 2}(z,r)= 1 $ is achieved for the function 
		$f_a$. 
	\end{enumerate}
\end{thm}
\subsection{Improved Bohr inequality for the class $\mathcal{B}$}
 Let $f$ be holomorphic in $\mathbb D$, and for $0<r<1$,  let $\mathbb D_r=\{z\in \mathbb C: |z|<r\}$.
Throughout the paper,  $S_r:=S_r(f)$ denotes the planar integral
\begin{align*}
	S_r:=\int_{\mathbb D_r} |f'(z)|^2 d A(z).
\end{align*}
 Kayumov and Ponnusamy \cite{Ponnusamy-2017} proved  the following sharp inequality  for functions in the class 
$ \mathcal{B}$, for the function $f(z)=\sum_{n=0}^{\infty}a_nz^n$,  
\begin{align}\label{Eqn-2.3}
	\frac{S_r}{\pi}:= \sum_{n=1}^\infty n|a_n|^2 r^{2n}\leq r^2\frac{(1-|a_0|^2)^2}{(1-|a_0|^2r^2)^2}\; \mbox{for}\;0<r\leq1/\sqrt{2}.
\end{align}
In \cite{Kayumov-CRACAD-2018},  Kayumov and Ponnusamy have established the following improved version of Bohr's inequality in terms of $ S_r $. In fact, it is worth mentioning that this result later play a significant role in studying improved Bohr inequalities for certain classes of harmonic mappings (see \cite{Ahamed-AMP-2021,Ahamed-CVEE-2021, Ahamed-CMFT1-2022, Ahamed-RMJM-2021}), operator-valued functions (see \cite{Allu-CMB-2022}).
\begin{thm}\cite{Kayumov-CRACAD-2018}\label{th-1.12}
	Suppose that  $ f(z)=\sum_{n=0}^{\infty}a_nz^n\in\mathcal{B} .$ Then 
	\begin{align}\label{Eqn-2.8}
		B_0(f,r)+ \frac{16}{9} \left(\frac{S_{r}}{\pi}\right) \leq 1 \quad \mbox{for} \quad r \leq \frac{1}{3},
	\end{align}
	and the numbers $1/3$, $16/9$ cannot be improved. Moreover, 
	\begin{align}\label{Eqn-2.9}
		|a_{0}|^{2}+B_1(f,r)+ \frac{9}{8} \left(\frac{S_{r}}{\pi}\right) \leq 1 \quad \mbox{for} \quad r \leq \frac{1}{2},
	\end{align}
	and the numbers $1/2$, $9/8$ cannot be improved.
\end{thm}
Ismagilov \emph{et al.} \cite{Ismagilov-2020-JMAA} proved the next result showing an interesting fact that if $ |a_0| $ is replaced by $ |f(z)| $ in equation (\ref{Eqn-2.8}) of Theorem \ref{th-1.12}, then the Bohr radius will decrease.
\begin{thm}	\cite{Ismagilov-2020-JMAA}\label{th-1.14}
Suppose that  $ f(z)=\sum_{n=0}^{\infty}a_nz^n\in\mathcal{B} $. Then 
\begin{align*}
	|f(z)|+B_1(f,r)+p\left(\frac{S_r}{\pi}\right)\leq 1\;\; \mbox{for}\; |z|=r\leq r_0=\sqrt{5}-2,
\end{align*}
where the constants $ r_0=\sqrt{5}-2\approx 0.236068 $ and $ p=2(\sqrt{5}-1) $ are sharp.
\end{thm} 
Continuing the research, Lui \emph{et al.} (see \cite{Liu-Liu-Ponnusamy-2021}) have refined the inequalities (\ref{Eqn-2.8}) and  (\ref{Eqn-2.9}) and established the following result.
\begin{thm}\label{th-1.19} \cite{Liu-Liu-Ponnusamy-2021}
Suppose that $ f(z)=\sum_{n=0}^{\infty}a_nz^n\in\mathcal{B} $. Then
\begin{align}\label{Eqn-2.10}
	B_0(f,r)+A(f_0, r)+\frac{8}{9}\left(\frac{S_r}{\pi}\right)\leq 1 \quad \mbox{for} \quad r \leq \frac{1}{3},
\end{align}
and the numbers $1/3$ and $8/9$ cannot be improved. Moreover,
\begin{align}\label{Eqn-2.11}
	|a_0|^2+B_1(f,r)+A(f_0, r)+\frac{9}{8}\left(\frac{S_r}{\pi}\right)\leq 1
\end{align}
for $ r\leq {1}/{(3-|a_0|)}, $ and the constant $ 9/8 $ cannot be improved.
\end{thm}
Recently, the inequality \eqref{Eqn-2.11} in \cite{Ahamed-CMB-2023} have improved by replacing the coefficient $|a_0|^2$ by $|f(z)|^2$ by obtaining the following result.
\begin{thm}\label{Thm-2.11} \cite{Ahamed-CMB-2023}
	Suppose that $ f(z)\in\mathcal{B} $ and $f(z)=\sum_{n=0}^{\infty}a_nz^n$. Then
	\begin{align*}
		|f(z)|^2+B_1(f,r)+A(f_0, r)+\lambda\left(\frac{S_r}{\pi}\right)\leq 1   
	\end{align*} for $|z|= r \leq 1/3,$ and the constant $\lambda=8/9$ cannot be improved.
\end{thm}
The investigation into improving the Bohr inequalities and their different sharp versions constitutes a thriving realm of the research in Bohr phenomenon. Actually, researchers have formulated several improved versions of the Bohr inequalities by incorporating the quantities $S_r/\pi$ and $S_r/(\pi-S_r)$ in the Bohr inequalities. This line of inquiry was initiated by Kayumov and Ponnusamy (see \cite[Theorem 1]{Kayumov-CRACAD-2018}), who established improved version of the Bohr inequalities, denoted as $B_0(f,r)$ and $|a_0|^2+B_1(f,r)$, by the incorporation of $S_r/\pi$ and $S_r/(\pi-S_r)$, respectively. Subsequently, Ismagilov \emph{et al.} (see \cite[Theorem 3]{Ismagilov-2020-JMAA}) further improved the primary inequality of Kayumov and Ponnusamy (refer to \cite[Theorem 1]{Kayumov-CRACAD-2018}) by substituting the initial coefficient $|a_0|$ with $|f(z)|$. Following this, Liu \emph{et al.} (see \cite[Theorem 4]{Liu-Liu-Ponnusamy-2021}) not only refined both inequalities introduced by Kayumov and Ponnusamy (in \cite[Theorem 1]{Kayumov-CRACAD-2018}) but also included an augmenting term $A(f_0, r)$ into the inequalities, thereby ushering in a significant advancement in this domain of study.\vspace{1.2mm}

As a result of a comprehensive scrutiny of the results and the corresponding proofs furnished by Liu \emph{et al.} (as mentioned in \cite{Liu-Liu-Ponnusamy-2021}), a compelling question emerges. 
Addressing this question could effectively bridge a gap within the existing research.
\begin{ques}\label{Q-3.3}
	Can the inequality (\ref{Eqnnn-2.3}), established by Liu \emph{et al.} (as elaborated in \cite{Liu-Liu-Ponnusamy-2021}), be further improved by the inclusion of the term $S_r/\pi$? If so, what can we say about its sharpness?
\end{ques}
We obtain the following result which addresses the Question \ref{Q-3.3} completely.
\begin{thm}\label{th-2.1}
	Suppose that $ f(z)=\sum_{n=0}^{\infty}a_nz^n\in\mathcal{B}.$ Then 
	\begin{align*}
		\mathcal{
			G}_{1,f}(r)&:=|f(z)|+|f^\prime(z)|r+B_2(f,r)+A(f_0, r)+\lambda\left(\frac{S_r}{\pi}\right)\leq 1
	\end{align*}
	for $ |z|=r\leq(\sqrt{17}-3)/4 $, where the numbers  $\lambda=(221-43\sqrt{17})/64$ and $(\sqrt{17}-3)/4$ cannot improved. The equality $ \mathcal{
		G}_{1,f}(r)=1 $ is achieved for the function 
	$f_a$.
\end{thm}
Moving forward, we will delve into the evaluation of the upper bound for the quantity $S_r/(\pi-S_r)$. It's worth noting that Ismagilov \emph{et al.} \cite{Ismagilov-2021-JMS} have introduced an improved version of the Bohr inequality by recognizing that
\begin{align}\label{Eqn-2.4}
	1-\frac{S_r}{\pi}\geq\frac{(1-r^2)(1-r^2|a_0|^4)}{(1-|a_0|^2r^2)^2}.
\end{align} 
With the help of \eqref{Eqn-2.3} and \eqref{Eqn-2.4}, the following inequality 
\begin{align}\label{Eq-3.8}
\frac{S_r}{\pi-S_r}\leq \frac{r^2(1-|a_0|^2)^2}{(1-r^2)(1-r^2|a_0|^4)}.
\end{align} 
can be obtained easily.\vspace{1.2mm}

 In view of this new setting, Ismagilov \emph{et al.} \cite{Ismagilov-2021-JMS} have investigated Theorem \ref{th-1.12} further and obtained the following sharp result.
\begin{thm} \label{th-1.18} \cite{Ismagilov-2021-JMS}
	Suppose that $ f(z)=\sum_{n=0}^{\infty}a_nz^n\in\mathcal{B} $. Then 
	\begin{align*}
		B_0(f,r)+ \frac{16}{9} \left(\frac{S_{r}}{\pi-S_r}\right) \leq 1 \quad \mbox{for} \quad r \leq \frac{1}{3},
	\end{align*}
	and the number $16/9$ cannot be improved. Moreover, 
	\begin{align*}
		|a_{0}|^{2}+B_1(f,r)+ \frac{9}{8} \left(\frac{S_{r}}{\pi-S_r}\right) \leq 1 \quad \mbox{for} \quad r \leq \frac{1}{2},
	\end{align*}
	and the number $9/8$ cannot be improved.
\end{thm}
In recent developments, an improved version of the inequality \eqref{Eqn-2.10} mentioned in \cite{Ahamed-CMB-2023} has been put forth. This improvement involves the substitution of the coefficient $|a_0|$ with $|f(z)|^2$, while employing the context of $S_r/(\pi-S_r)$ instead of $S_r/\pi$, as detailed below.
\begin{thm}\label{Thm-2.14} \cite{Ahamed-CMB-2023}
Suppose that $ f(z)\in\mathcal{B} $ and $f(z)=\sum_{n=0}^{\infty}a_nz^n$. Then
\begin{align}\label{Eq-3.9}
	|f(z)|^2+B_1(f,r)+A(f_0, r)+\lambda\left(\frac{S_r}{\pi-S_r}\right)\leq 1, 
\end{align} for $|z|= r \leq 1/3,$ and the constant $\lambda=8/9$ cannot be improved.
\end{thm}
It's worth observing that in all the aforementioned discussions, when scholars employ the context of $S_r/(\pi-S_r)$ in place of $S_r/\pi$, the parameter $\lambda$ remains consistent. To facilitate further exploration of the subject, the following question naturally arises.
\begin{ques}\label{Q-3.4}
Can we improve the inequality (\ref{Eqnnn-2.3}) which was established by Lui \emph{et al.} (see \cite{Liu-Liu-Ponnusamy-2021}) by adding $S_r/(\pi-S_r)$ instead of $S_r/\pi$ and it will interesting to see that the same $\lambda$ will work or not?
\end{ques}
We prove the following sharp result corresponding to Theorem \ref{th-1.19} and it completely addresses the Question \ref{Q-3.4}.
\begin{thm}\label{th-2.2}
Suppose that $ f(z)=\sum_{n=0}^{\infty}a_nz^n\in\mathcal{B} $ and $S_r$ denotes the Riemann surface of the function $f^{-1}$ defined on the image of the sub-disk $|z|\leq r$ under the mapping $f$. Then 
\begin{align*}
\mathcal{G}_{2,f}(r)&:=|f(z)|+|f^\prime(z)|r+B_2(f,r)+A(f_0, r)+\lambda\left(\frac{S_r}{\pi-S_r}\right)\leq 1
\end{align*}
for $ r\leq(\sqrt{17}-3)/4 $, where the numbers  $\lambda=(221-43\sqrt{17})/64$ and $(\sqrt{17}-3)/4\;$ cannot be improved. The equality $ \mathcal{G}_{2,f}(r)=1 $ is achieved for the function 
$f_a$.
\end{thm}
 \begin{rem}
 The values of $\lambda=(221-43\sqrt{17})/64$ and $(\sqrt{17}-3)/4$ remain invariant in Theorems \ref{th-2.1} and \ref{th-2.2}, thereby exemplifying the alignment of our study with the core tenets of Bohr inequalities for the class $\mathcal{B}$.
 \end{rem}
\subsection{Key Lemmas, Proof of Theorems \ref{Thmm-2.3} and \ref{th-2.4}}
In this section, we present some necessary lemmas which will play key roles to prove Theorems \ref{Thmm-2.3} and \ref{th-2.4} for the class $\mathcal{B}$.
\begin{lem} \cite{Kayumov-CMFT-2017} \label{lem-3.1}
	Let $|a_0|<1$ and $0<r\leq1$. Suppose that $f(z)=\sum_{n=0}^{\infty}a_nz^n\in\mathcal{B} $. Then the following sharp inequality holds:
	\begin{align*}
		\sum_{n=1}^{\infty}|a_n|^2r^{pn}\leq r^p\dfrac{(1-|a_0^2|)^2}{1-|a_0|^2r^p}.
	\end{align*}
\end{lem}
\begin{lem} \cite{Ponnusamy-2017}\label{lem-3.2}
	Suppose that $ f(z)=\sum_{n=0}^{\infty}a_nz^n \in\mathcal{B}.$ Then the following sharp inequality holds:
	\begin{align*}
		\frac{S_r}{\pi}=\sum_{n=1}^{\infty}n|a_n|^2r^{2n}\leq r^2\frac{(1-|a_0|^2)^2}{(1-|a_0|^2r^2)^2}\; \,\,\mbox{for}\;\; 0<r\leq1/\sqrt{2}.
	\end{align*}
\end{lem}

\begin{lem}\cite{Liu-Liu-Ponnusamy-2021} \label{lem-3.5} 
	Suppose that $ f(z)=\sum_{n=0}^{\infty}a_nz^n\in\mathcal{B} $. Then for any N$\in\mathbb{N}$, the following inequality holds:
	\begin{align*}
		\sum_{n=N}^{\infty}&|a_n|r^n+sgn(t)\sum_{n=1}^{t}|a_n|^2\dfrac{r^N}{1-r}+\left(\dfrac{1}{1+|a_0|}+\dfrac{r}{1-r}\right)\sum_{n=t+1}^{\infty}|a_n|^2r^{2n}\\&\leq (1-|a_0|^2)\dfrac{r^N}{1-r}, 
	\end{align*}
	\;\;\mbox{for}\;\; $ r\in[0,1),$ where $t=\lfloor{(N-1)/2}\rfloor$.
\end{lem}
\begin{lem} \cite{Graham-2003}\label{lem-3.8}
	Suppose that $ f(z)=\sum_{n=0}^{\infty}a_nz^n\in\mathcal{B} $. Then  $|a_n|\leq 1-|a_0|^2$ \;\; \mbox{for all}\;\; $n=1,2,...$
\end{lem}
\begin{lem} \cite{Dai-Proc-2008} \label{lem-3.10}
	Suppose that $ f(z)=\sum_{n=0}^{\infty}a_nz^n\in\mathcal{B} $. Then, for all $n=1,2,3.....$, we have
	\begin{align*}
		|f^{(k)}(z)|\leq \frac{k!(1-|f(z)|^2)}{(1-|z|^2)^k}(1+|z|)^{k-1}, \;\; |z|<1.
	\end{align*}
	\end{lem}
\begin{proof}[\bf Proof of Theorem \ref{Thmm-2.3}]
	Let $f(z)=\sum_{k=0}^{\infty}a_kz^k$$\in\mathcal{B}$ and $a:=|a_0|$ . Since $f(0)=a_0 $, by the Schwarz-Pick Lemma, for $z\in\mathbb{D},$ we easily obtain that 
	\begin{align}\label{Eq-2.14}
		\dfrac{|f(z)-a_0|}{|1-\overline{a_0}f(z)|}\leq|z|,\;\mbox{and}\;  	|f^{(k)}(z)|\leq \frac{k!(1-|f(z)|^2)}{(1-|z|^2)^k}(1+|z|)^{k-1}, \;\; |z|<1.
	\end{align}
	Thus it follows from the above inequalities and Lemma \ref{lem-3.8} that, for $z=re^{i\theta}\in\mathbb{D}$,
	\begin{align}\label{Eq-2.15}
		|f(z)|\leq\dfrac{r+a}{1+ra} \;\mbox{and}\;|a_k|\leq1-a^2 \;\mbox{for}\; k=1,2,\cdots
	\end{align}
In view of the inequalities \eqref{Eq-2.14}, \eqref{Eq-2.15}and Lemma \ref{lem-3.5} (with $N=3$), a simple computation shows that 
	\begin{align*}
		\mathcal{D}_{f}(z,r)&\leq |f(z)|^2+\frac{r}{1-r^2}\left(1-|f(z)|^2\right)+\frac{r^2(1+r)}{(1-r^2)^2}\left(1-|f(z)|^2\right)+B_3(f,r)+\frac{|a^2_1|r^3}{1-r}\\&\quad+\frac{(1-a^2)r^3}{(1-r)}\\&\leq\left(\frac{r+a}{1+ra}\right)^2+\frac{r}{1-r^2}\left(1-\left(\frac{r+a}{1+ra}\right)^2\right)+\frac{r^2(1+r)}{(1-r^2)^2}\left(1-\left(\frac{r+a}{1+ra}\right)^2\right)\\&\quad+\left(1-a^2\right)\frac{r^3}{1-r}\\&	\leq\left(\frac{r+a}{1+ra}\right)^2+\frac{(1-a^2)r}{(1+ar)^2}+\frac{(1-a^2)r^2}{(1-r)(1+ar)^2}+\frac{(1-a^2)r^3}{(1-r)}\\&	=1+\dfrac{(1-a)G_1(a,r)}{(1-r)^2(1+ar)^2},
	\end{align*} 
	where $	G_1(a,r):=-1+2r+r^2+2ar^4+a^2r^5$.	For each fixed $r\in[0,1]$, it is clear that  $G_1(a,r)$ is a monotonically increasing function of $a\in[0,1)$ and thus, 
	$	G_1(a,r)\leq G_1(1,r)=-1+2r+r^2+2r^4+r^5=G_1(r)$. Therefore, $\mathcal{G}_f(z)\leq1$ if $G_1(r)\leq 0$. In fact, we see that $G_1(r)\leq 0$ for $r\leq\gamma_1$, where $r_0$ is the unique root in $(0,1)$ of the equation $G_1(r)=-1+2r+r^2+2r^4+r^5=0$. Consequently, the inequality of the result is obtained.\vspace{1.2mm}
	
	 The second part of the proof is to show that $r_0$ is best possible. In order to do that, we consider the function $f_a$ defined by 
	\begin{align}\label{EQ-4.1}
	f_a(z)=\dfrac{a-z}{1-az}=a-(1-a^2)\sum_{k=1}^{\infty}a^{k-1}z^k.
	\end{align}
	 For this function $ f_a $ with $ z=-r $, an easy computation shows that 
	\begin{align*}
		\mathcal{D}_{f_a}(z,r)&= |f_a(-r)|^2+|f_a^\prime(-r)|r+\frac{|f_a^{\prime\prime}(-r)|}{2!}r^2+B_3(f_a,r)+|a^2_1|\frac{r^3}{1-r}+A(f_1,r)\\&=\left(\frac{r+a}{1+ra}\right)^2+\frac{r(1-a^2)}{(1+ra)^2}+\frac{r^2a(1-a^2)}{(1+ra)^3}+\sum_{n=3}^{\infty}(1-a^2)a^{n-1}r^n\\&\quad+\frac{r^3(1-a^2)^2}{1-r}+\frac{(1+ra)}{(1+a)(1-r)}\sum_{n=2}^{\infty}\left((1-a^2)a^{n-1}\right)^2r^{2n}\\&=\left(\frac{r+a}{1+ra}\right)^2+\frac{r(1-a^2)}{(1+ra)^2}+\frac{r^2a(1-a^2)}{(1+ra)^3}+\frac{r^3a^2(1-a^2)^2}{(1-ar)}+\frac{r^3(1-a^2)^2}{(1-r)}\\&\quad+\frac{(1+ra)}{(1+a)(1-r)}\frac{(1-a^2)^2a^2r^4}{(1-a^2r^2)}\\&=1+\dfrac{(1-a^2)E_1(a,r)}{(1-r)(1+ar)^3}>1 
	\end{align*}
	if, and only if, $(1-a^2)E_1(a,r)>0$, where $E_1(a,r)$ is defined by 
	\begin{align*}
	E_1(a,r)&:=-1+2r-ar+3ar^2-ar^3+2ar^4+3a^2r^5+a^3r^6\\&\leq-1+2r+3ar^2-ar^3+2ar^4+3a^2r^5+a^3r^6=E^{*}_1(a,r).
	\end{align*} 
Indeed, it is evident that
	\begin{align*}
	\lim_{a\to 1}E^{*}_1(a,r)=-1+2r+3r^2-r^3+2r^4+3r^5+r^6=E_1(r).
	\end{align*}
	Since $E_1(r)$ is a real-valued differentiable function on $(0,1)$ satisfying $E_1(0)=-1<0$ and $E_1(0)=9>0$ and $E_1(r)$ is increasing function in $(0,1)$, then with the help of Intermediate value theorem, we can say that $E_1(r)$ has unique real root  $r_1\approx0.333004$ in $(0,1)$ and $ r_1<r_0$. Thus, it is clear that $E_1(r)>0$ for $r>r_1$ and hence, $E_1(a,r)>0$ for all $r>r_0$ and $a\in[0,1),$ and this shows that $r_0$ is the best possible. This proves the sharpness.
\end{proof}	

\begin{proof}[\bf Proof of Theorem \ref{th-2.4}]
	Let $f(z)=\sum_{k=0}^{\infty}a_kz^k$$\in\mathcal{B}$ and $a:=|a_0|$. Since $f(0)=a_0$, In view of the inequalities \eqref{Eq-2.14}, \eqref{Eq-2.15}and Lemma \ref{lem-3.5} (with $N=4$), a simple computation shows that 
	\begin{align*}
		\mathcal{J}_{f, 1}(z,r)&\leq \frac{r+a}{1+ra}+\frac{r}{1-r^2}\left(1-\left(\frac{r+a}{1+ra}\right)^2\right)+\frac{r^2(1+r)}{(1-r^2)^2}\left(1-\left(\frac{r+a}{1+ra}\right)^2\right)\\&\quad+\frac{r^3(1+r)^2}{(1-r^2)^3}\left(1-\left(\frac{r+a}{1+ra}\right)^2\right)+\left(1-a^2\right)\frac{r^4}{1-r}\\&	\leq\frac{r+a}{1+ra}+\frac{(1-a^2)r}{(1+ar)^2}+\frac{(1-a^2)r^2}{(1-r)(1+ar)^2}+\frac{(1-a^2)r^3}{(1-r)^2(1+ar)^3}+\frac{(1-a^2)r^4}{(1-r)}\\&	=1+\dfrac{(1-a)G_2(a,r)}{(1-r)^2(1+ar)^2},
	\end{align*}where 
 \begin{align*}
	G_2(a,r):&=r^6(1-r)a^3+[r^5\{(1-r)+(1-r^2)\}]a^2+[r^5+2r^4\{(1-r)\\&\quad\quad+(1-r^2)\}]a-r^5+r^4+2r^3-4r^2+4r-1.
\end{align*}
	For each fixed r$\in[0,1]$, it can be shown that  $G_2(a,r)$ is a monotonically increasing function of $a\in[0,1)$ and thus, we have
	\begin{align*}
		G_2(a,r)\leq G_2(1,r)=-1+4r-2r^2+3r^4+2r^5-2r^6-2r^7:=G_2(r).
	\end{align*}
	Therefore, we see that $\mathcal{J}_{f, 1}(z,r)\leq 1$ if $G_2(r)\leq 0$. However, a simple computation yields that $G_2(r)\leq 0$ for $r\leq R_1$, where $R_1$ is the unique root in $(0,1)$ of the equation $ G_2(r)=0 $. Thus the desired inequality of the theorem is established.\vspace{1.2mm}	
	
	The second part of the proof is to show that the radius $R_1$ is optimal. In order to do that, we consider the function $f_a$  defined by (\ref{EQ-4.1}). For such $f_a$ with $z=-r$, an easy computation shows that
	\begin{align*}
		\mathcal{J}_{f_a, 1}(z, r)&= |f_a(-r)|+|f_a^\prime(-r)|r+\frac{|f_a^{\prime\prime}(-r)|}{2!}r^2+\frac{|f_a^{\prime\prime\prime}(-r)|}{3!}r^3+B_4(f_a,r)+|a_1|^2\frac{r^4}{1-r}\\&\quad+A(f_1, r)\\&=\frac{r+a}{1+ra}+\frac{r(1-a^2)}{(1+ra)^2}+\frac{r^2a(1-a^2)}{(1+ra)^3}+\frac{r^3a^2(1-a^2)}{(1+ra)^4}+\sum_{n=4}^{\infty}(1-a^2)a^{n-1}r^n\\&\quad+\frac{r^4(1-a^2)^2}{1-r}+\frac{(1+ra)}{(1+a)(1-r)}\sum_{n=2}^{\infty}\left((1-a^2)a^{n-1}\right)^2r^{2n}\\&=\frac{r+a}{1+ra}+\frac{r(1-a^2)}{(1+ra)^2}+\frac{r^2a(1-a^2)}{(1+ra)^3}+\frac{r^3a^2(1-a^2)}{(1+ra)^4}+\frac{r^4(1-a^2)^2}{(1-r)}\\&\quad+\frac{r^4a^3(1-a^2)^2}{(1-ar)}+\frac{(1+ra)}{(1+a)(1-r)}\frac{(1-a^2)^2a^2r^4}{(1-a^2r^2)}\\&=1+\dfrac{(1-a)E_2(a,r)}{(1-r)(1+ar)^4}>1 
	\end{align*}if, and only if, $(1-a)E_2(a,r)>0$, where $E_2(a,r)$ is defined by 
\begin{align*}
E_2(a,r)&:=-1+3r-2ar-2r^2+8ar^2-6ar^3+6a^2r^3+2a^3r^3+r^4+ar^4-6a^2r^4\\&\quad-a^3r^4+4ar^5+4a^2r^5-a^3r^5+6a^2r^6+6a^3r^6+4a^3r^7+4a^4r^7+a^4r^8+a^5r^8\\&\leq-1+3r-2ar+8ar^2-6ar^3+6a^2r^3+2a^3r^3+r^4+ar^4-a^3r^4+4ar^5\\&\quad+4a^2r^5-a^3r^5+6a^2r^6+6a^3r^6+4a^3r^7+4a^4r^7+a^4r^8+a^5r^8\\&:=E^{*}_2(a,r).
\end{align*} 
It is easy to see that
\begin{align*}
\lim_{a\to 1}E^{*}_2(a,r)=-1+r+8r^2+2r^3+r^4+7r^5+12r^6+8r^7+2r^8:=E_2(r).
\end{align*}
Since, $E_2(r)$ is a real-valued continuously differentiable  function on $(0,1)$ satisfying $E_2(0)=-1<0$ and  $E_2(0)=40>0$ and $E_2(r)$ is increasing function in $(0,1)$, with the help of Intermediate value theorem, it is evident that $E_2(r)$ has unique real $r_2\approx 0.283682$ root in $(0,1)$ and $r_2<R_1$. Thus, it is clear that $E_2(r)>0$ for $r>R_1$, Hence, it is easy to see that $E_2(a,r)>0$ for all $r>R_1$ and $a\in[0,1).$ This shows that $R_1$ is the best possible. This proves the sharpness.\vspace{1,2mm}

By the similar method as applied above, the other part can be proved easily. Hence, we omit the details of its proof.
\end{proof}	
\begin{proof}[\bf Proof of Theorem \ref{th-2.1}]
    Let $|a_0|=a\in(0,1)$. Combining the Lemma \ref{lem-3.5}	(with $N=2 $) with the classical inequality for $|f(z)|$ and Schwarz-Pick Lemma, we obtain the following estimate
		 \begin{align*}
		 	\mathcal{
		 	G}_{1,f}(r)&:=|f(z)|+|f^\prime(z)|r+B_2(f,r)+A(f_0, r)+\frac{221-43\sqrt{17}}{64}\left(\frac{S_r}{\pi}\right)\\& \leq \frac{r+a}{1+ra}+\frac{r}{1-r^2}\left(1-\left(\frac{r+a}{1+ra}\right)^2\right)+\frac{(1-a^2)r^2}{1-r}\\&\quad+\frac{221-43\sqrt{17}}{64}\frac{(1-a^2)^2r^2}{(1-a^2r^2)^2}:=A_1(r).
		 \end{align*}
		 It is easy to see that $A_1(r)$ is an increasing function of $r$. Thus it follows that
		 \begin{align*}
		 	A_1(r)&\leq A_1\left(\frac{\sqrt{17}-3}{4}\right)\\&=\frac{(-3+\sqrt{17})+4a}{4+\left(-3+\sqrt{17}\right)a}+\frac{(-3+\sqrt{17})^2(1-a^2)}{4(7-\sqrt{17})}+\frac{4(-3+\sqrt{17})(1-a^2)}{\left(4+(-3+\sqrt{17})|a_0|\right)^2}\\&\quad+\frac{(-3+\sqrt{17})^2(221-43\sqrt{17})(1-a^2)^2}{16\left(8+(-13+3\sqrt{17})a^2\right)^2}\\&=1-\frac{(1-a)F_1(a)}{(7-\sqrt{17})\left(-16+(26-6\sqrt{17})a^2\right)^2},
		 \end{align*} 
	 where \begin{align*}
		 	F_1(a)&:=(-23126+5658\sqrt{17})+2\left(-28769+6983\sqrt{17}\right)a+24\left(-2041+495\sqrt{17}\right)a^2\\&\quad+8\left(-2041+495\sqrt{17}\right)a^3.
		 \end{align*} 
	 Moreover, it is not hard to show that $F_1(a)\geq 0$ $\forall\; a\in (0,1)$. Consequently, we have $A_1\left({\sqrt{17}-3}/{4}\right)\leq 1$, and with this the desired inequality of the result is established.\vspace{1.2mm}
	 
To prove the constant $(221-43\sqrt{17})/64$ is best possible, let us consider the function defined as 
\begin{align}\label{EQ-4.2}
f^*_a(z):=\frac{a+z}{1+az}=a+(1-a^2)\sum_{k=1}^{\infty}(-a)^{k-1}z^k.
\end{align}
For this function $f^*_a$ with $z=r$, a straightforward calculation shows that 
\begin{align*}
\mathcal{G}_{1,	f^*_a}(r)&:=|f^*_a(r)|+|	\left(f^*_a\right)^{\prime}(r)|r+B_2(f,r)+A(({f}^*_a)_\theta,r)+\lambda\left(\frac{S_r}{\pi}\right)\\&=\frac{r+a}{1+ra}+\frac{(1-a^2)r}{(1+ar)^2}+\frac{(1-a^2)ar^2}{(1-ar)}+\frac{(1+ar)}{(1+a)(1-r)}\frac{(1-a^2)^2r^2}{1-a^2r^2}\\&\quad+\lambda\frac{(1-a^2)^2r^2}{(1-a^2r^2)^2}.
\end{align*}
In the case $r=(\sqrt{17}-3)/4$, the last expression becomes 
\begin{align}\label{Eqn-1.1}
1+\frac{(1-a)^2F_2(a,\lambda)}{(7-\sqrt{17})\left(8+(-13+3\sqrt{17})a^2\right)^2},
\end{align}
where
\begin{align*}
F_2(a,\lambda)&:=(-1248+288\sqrt{17})+(-5456+1328\sqrt{17})a+(-9168+2224\sqrt{17})a^2\\&\quad+(-4082+990\sqrt{17})a^3+(-4082+990\sqrt{17})a^4+16\lambda(71-17\sqrt{17})(1+a)^2.
\end{align*}
The expression \eqref{Eqn-1.1} is bigger then $1$ when $F_2(a,\lambda)>0$ for $a\to 1^{-}$ if
\begin{align*}
\lambda >\lim_{a\to 1^{-}}\frac{M_1(a)}{16(71-17\sqrt{17})(1+a)^2}=\frac{221-43\sqrt{17}}{64},
\end{align*} where $M_1(a)=1248-288\sqrt{17}+(5456-1328\sqrt{17})a+(9168-2224\sqrt{17})a^2+(4082+990\sqrt{17})a^3+(4082-990\sqrt{17})a^4$. This shows that the constant $(221-43\sqrt{17})/64$ is best possible and with this the the proof is completed.
\end{proof}
\begin{proof}[\bf Proof of Theorem \ref{th-2.2}]
We suppose that $|a_0|=a\in(0,1)$. Combining, Lemma \ref{lem-3.5}	(with $N=2$) with  the inequality for $|f(z)|$ and Schwarz-Pick lemma, in view of \eqref{Eq-3.8}, we obtain
\begin{align*}
\mathcal{G}_{2,f}(r)&:=|f(z)|+|f^\prime(z)|r+B_2(f,r)+A(f_0, r)+\frac{221-43\sqrt{17}}{64}\left(\frac{S_r}{\pi-S_r}\right)\\
&\leq \frac{r+a}{1+ra}+\frac{r}{1-r^2}\left(1-\left(\frac{r+a}{1+ra}\right)^2\right)+\frac{(1-a^2)r^2}{1-r}+\frac{221-43\sqrt{17}}{64}\times\\&\quad\frac{(1-a^2)^2r^2}{(1-r^2)(1-a^4r^2)^2}:=A_2(r).
\end{align*}
It is easy to see that $A_2(r)$ is an increasing function of $r$. Consequently, we see that
\begin{align*}
A_2(r)&\leq A_2\left((\sqrt{17}-3)/4\right)\\&=\frac{(-3+\sqrt{17})+4a}{4+\left(-3+\sqrt{17}\right)a}+\frac{(-3+\sqrt{17})^2(1-a^2)}{4(7-\sqrt{17})}+\frac{4(-3+\sqrt{17})(1-a^2)}{\left(4+(-3+\sqrt{17})a\right)^2}\\&\quad+\frac{(-3+\sqrt{17})^2(221-43\sqrt{17})(1-a^2)^2}{16\left(-5+3\sqrt{17})\right)(8+(-13+3\sqrt{17})a^4)}\\&=1-\frac{(1-a)^3F_3(a)}{2(4+(-3+\sqrt{17})a)^2(52\sqrt{17}-172+(611-149\sqrt{17})a^4)},
\end{align*} 
where
	\begin{align*}
		F_3(a)&=(-77164+18804\sqrt{17})+8\left(-1133+283\sqrt{17}\right)a+\left(134385-32567\sqrt{17}\right)a^2\\&\quad+\left(-21913+5327\sqrt{17}\right)a^3+4\left(-33267+8069\sqrt{17}\right)a^4\\&\quad+4\left(-17725+4299\sqrt{17}\right)a^5.
	\end{align*} It can be easily shown that $F_3(a)\geq 0$ $\forall\; a\in (0,1)$. Moreover, it is easy to see that $A_2\left((\sqrt{17}-3)/4\right)\leq 1$, and with this the desired inequality of the result established.\vspace{1.2mm} 

To prove the constant $(221-43\sqrt{17})/64$ is best possible, let us consider the function $f^{*}_a$ which is given by \eqref{EQ-4.2}. For this function, a straightforward calculation shows that
	\begin{align*}
		\mathcal{G}_{2,f^{*}_a}(r)&:=|f^{*}_a(r)|+|(f^{*}_a)^\prime(r)|r+B_2(f^{*}_a,r)+A(({f}^*_a)_\theta,r)+\lambda\left(\frac{S_r}{\pi-S_r}\right)\\&=\frac{r+a}{1+ra}+\frac{(1-a^2)r}{(1+ar)^2}+\frac{(1-a^2)ar^2}{(1-ar)}+\frac{(1+ar)}{(1+a)(1-r)}\frac{(1-a^2)^2r^2}{1-a^2r^2}\\&\quad+\frac{\lambda(1-a^2)^2r^2}{(1-r^2)(1-a^4r^2)^2}
	\end{align*}
In the case of when $r=(\sqrt{17}-3)/4$, the last expression becomes 
	\begin{align*}
			1+\frac{(1-a)^2F_4(a,\lambda)}{(4+(-3+\sqrt{17})a)^2(52\sqrt{17}-172+(611-149\sqrt{17})a^4)},
	\end{align*} where
	\begin{align*}
		F_4(a,\lambda)&=(10464-2592\sqrt{17})+16(611-149\sqrt{17})a+2(5588-1356\sqrt{17})a^2\\&\quad+24(-1397+339\sqrt{17})a^4+4(-7771+1885\sqrt{17})a^5\\&\quad+2(-17725+4299\sqrt{17})a^6+16\lambda(1+a)^2[(284-68\sqrt{17})\\&\quad+4(-251+61\sqrt{17})a+(895-217\sqrt{17})a^2].
	\end{align*}
	The expression (\ref{Eqn-1.1} ) is bigger then $1$ when $F_4(a,\lambda)>0$ and as $a\to1^{-}$ shows that
	\begin{align*}
		\lambda &>\lim_{a\to 1^{-}}\frac{M_2(a)}{16(1+a)^2[(284-68\sqrt{17})+4(-251+61\sqrt{17})a+(895-217\sqrt{17})a^2]}\\&=\frac{(221-43\sqrt{17})}{64},
	\end{align*}
	 where $M_2(a)=96(-109+27\sqrt{17})+16(-611+149\sqrt{17})a+8(-1397+339\sqrt{17})a^2+24(1397-339\sqrt{17})a^4+2(17725-4299\sqrt{17})a^6$. This shows that $\lambda>(221-43\sqrt{17})/64$ as $a\to 1^{-}$. This completes the proof.
\end{proof}	
\section{Generalization of the Bohr inequality for certain classes of harmonic mappings and its applications}
Let $ \mathcal{H}(\Omega) $ be the class of complex-valued functions harmonic in $ \Omega $. It is well-known that functions $ f $ in the class $ \mathcal{H}(\Omega) $ has the following representation $ f=h+\overline{g} $, where $ h $ and $ g $ both are analytic functions in $ \Omega $. The famous Lewy's theorem \cite{Lew-BAMS-1936} in $ 1936 $ states that a harmonic mapping $ f=h+\overline{g} $ is locally univalent on $ \Omega $ if, and only if, the determinant $ |J_f(z)| $ of its Jacobian matrix $ J_f(z) $ does not vanish on $ \Omega $, where
\begin{equation*}
	|J_f(z)|:=|f_{z}(z)|^2-|f_{\bar{z}}(z)|^2=|h^{\prime}(z)|^2-|g^{\prime}(z)|^2\neq 0.
\end{equation*}
In view of this result, a locally univalent harmonic mapping  is sense-preserving if $ |J_f(z)|>0 $ and sense-reversing if $|J_{f}(z)|<0$ in $\Omega$. For detailed information about the harmonic mappings, we refer the reader to \cite{Clunie-AASF-1984,Duren-2004}. In \cite{Kay & Pon & Sha & MN & 2018}, Kayumov \emph{et al.} first established the harmonic extension of the classical Bohr theorem, since then investigating on the Bohr-type inequalities for certain class of harmonic mappings becomes an interesting topic of research in geometric function theory.\vspace{1.2mm}

 Methods of harmonic mappings have been applied to study and solve the fluid flow problems (see \cite{Aleman-2012,Constantin-2017}). In particular, univalent harmonic mappings having special geometric
 properties such as convexity, starlikeness and close-to-convexity arise naturally in planar fluid dynamical problems. For example, in $2012$, Aleman and Constantin \cite{Aleman-2012} established a connection between harmonic mappings and ideal fluid flows. In fact, Aleman and Constantin have developed ingenious technique to solve the incompressible two dimensional Euler equations in terms of univalent harmonic mappings (see \cite{Constantin-2017} for details).\vspace{1.2mm}

Bohr phenomenon can be studied in view of the Euclidean distance and in this paper, we study the same for certain classes of harmonic mappings. Before we go into details, we recall here the following concepts. 
\begin{defi}
Let $ f $ and $ g $ be two analytic functions in the unit disc $ \mathbb{D} $. We say that $ g $ is subordinate to $ f $ if there is a function $ \varphi $, analytic in $ \mathbb{D} $, $ \varphi(\mathbb{D})\subset \mathbb{D} $ and $ \varphi(0)=0 $ so that $ g=f\circ \varphi $. In particular, when the function $ f $ is univalent, $ g $ is subordinate to $ f $ when $ g(\mathbb{D}\subset f(\mathbb{D})) $ and $ g(0)=f(0) $ (see \cite[p. 190]{Duren-1983}). Consequently, when $ g $ is subordinate to $ f $, $ |g^{\prime}(0)|\leq |f^{\prime}(0)| $. The class of all function $ g $ subordinate to a fixed function $ f $ is denoted by $ \mathcal{S}(f) $ and $ f(\mathbb{D})=\Omega $.
\end{defi}
\begin{defi}\cite{Abu-CVEE-2010}
	We say that $ \mathcal{S}(f) $ has Bohr phenomenon if for any $ g=\sum_{n=0}^{\infty}b_nz^n\in\mathcal{S}(f) $ and $ f=\sum_{n=0}^{\infty}a_nz^n $ there is a $ \rho^*_0 $, $ 0<\rho^*_0\leq 1 $ so that $ \sum_{n=1}^{\infty}|b_nz^n|\leq d(f(0), \partial \Omega)  $,  for $ |z|<\rho^*_0 $. Notice that $  d(f(0), \partial \Omega) $ denote the Euclidean distance between $ f(0) $ and the boundary of a domain $ \Omega $, $ \partial\Omega $. In particular, when $ \Omega=\mathbb{D} $, $  d(f(0), \partial \Omega)=1-|f(0)| $ and in this case $ \sum_{n=1}^{\infty}|a_nz^n|\leq d(f(0), \partial \Omega) $ reduces to $ \sum_{n=0}^{\infty}|a_nz^n|\leq 1 $.
\end{defi}
The classical Bohr inequality can be written as
\begin{equation}\label{e-1.2}
	d\left(\sum_{n=0}^{\infty}|a_nz^n|,|a_0|\right)=\sum_{n=1}^{\infty}|a_nz^n|\leq 1-|f(0)|=d(f(0),\partial (\mathbb{D})),
\end{equation}
where $ d $ is the Euclidean distance. More generally, a class $ \mathcal{F} $ of analytic functions $ f(z)=\sum_{n=0}^{\infty}a_nz^n $ mapping $ \mathbb{D} $ into a domain $ \Omega $ is said to satisfy a Bohr phenomenon if an inequality of type \eqref{e-1.2} holds uniformly in $ |z|\leq \rho_0 $, where $ 0<\rho_0\leq 1 $ for functions in the class $ \mathcal{F} $. Similar definition makes sense for harmonic functions (see \cite{Kay & Pon & Sha & MN & 2018}).\vspace{1.2mm}

In \cite{Abu-CVEE-2010}, Abu-Muhanna have obtained the following result for subordination class $ \mathcal{S}(f) $ when $ f $ is univalent.
\begin{thm}\cite{Abu-CVEE-2010}
	If $ g=\sum_{n=0}^{\infty}b_nz^n\in\mathcal{S}(f) $ and $ f=\sum_{n=0}^{\infty}a_nz^n $ is univalent, then 
	\begin{align*}
		\sum_{n=1}^{\infty}|b_n|r^n\leq d(f(0), \partial\Omega),
	\end{align*}
	for $ |z|=\rho^*_0\leq 3-\sqrt{8}\approx0.17157 $, where $ \rho^*_0 $ is sharp for the Koebe function $ f(z)=z/(1-z)^2 $.
\end{thm}

\par Let $ \mathcal{H} $ be the class of all complex-valued harmonic functions $ f=h+\bar{g} $ defined on the unit disk $ \mathbb{D} $, where $ h $ and $ g $ are analytic in $ \mathbb{D} $ with the normalization $ h(0)=h^{\prime}(0)-1=0 $ and $ g(0)=0 $. Let $ \mathcal{H}_0 $ be defined by $ 	\mathcal{H}_0=\{f=h+\bar{g}\in\mathcal{H} : g^{\prime}(0)=0\}. $ Therefore, each $f=h+\overline{g}\in \mathcal{H}_{0}$ has the following representation 
\begin{equation}\label{e-7.2}
	f(z)=h(z)+\overline{g(z)}=\sum_{n=1}^{\infty}a_nz^n+\overline{\sum_{n=1}^{\infty}b_nz^n}=z+\sum_{n=2}^{\infty}a_nz^n+\overline{\sum_{n=2}^{\infty}b_nz^n},
\end{equation}
where $ a_1 = 1 $ and $ b_1 = 0 $, since $ a_1 $ and $ b_1 $ have been appeared in later results and corresponding proofs.\vspace{1.2mm}

Let us recall the Bohr radius for the class of harmonic mappings.
\begin{defi}\cite{Kay & Pon & Sha & MN & 2018}
	Let $ f\in\mathcal{H}_0 $ be given by \eqref{e-7.2}. Then the Bohr phenomenon is to find a constant $ R^*\in (0, 1] $ such that the inequality $ r+\sum_{n=2}^{\infty}\left(|a_n|+|b_n|\right)r^n\leq d\left(f(0), \partial\Omega\right) $ holds for $ |z|=r\leq R^* $, where  $ d\left(f(0), \partial\Omega\right) $ is the Euclidean distance between $ f(0) $ and the boundary of $ \Omega:=f(\mathbb{D}) $. The largest such radius $ R^* $ is called the Bohr radius for the class $ \mathcal{H}_0 $.
\end{defi}
Let $\{\varphi_n(r)\}^{\infty}_{n=0}$ be a sequence of non-negative continuous function in $[0,1)$ such that the series $\sum_{n=0}^{\infty}\varphi_n(r)$ converges locally uniformly on the interval $[0,1)$.
\begin{thm}\cite{Kayumov-MJM-2022}\label{Th-3.2}
	Let  $ f\in\mathcal{B} $ with $f(z)=\sum_{n=0}^{\infty} a_{n}z^{n}$ and $p\in (0,2]$. If 
	\begin{align*}
		\varphi_0(r)>\frac{2}{p}\sum_{n=1}^{\infty}\varphi_n(r)\; \mbox{for}\; r \in [0,R),
	\end{align*}where $R$ is the minimal positive root of the equation $	\varphi_0(x)=\frac{2}{p}\sum_{n=1}^{\infty}\varphi_n(x),$ then the following sharp inequality holds:
	\begin{align*}
		B_f(\varphi,p,r):=|a_0|^p\varphi_0(r)+\sum_{n=1}^{\infty}|a_n|\varphi_n(r)\leq \varphi_0(r)\;\mbox{for all}\; r\leq R.
	\end{align*}In the case when $\varphi_0(r)<\frac{2}{p}\sum_{n=1}^{\infty}\varphi_n(r)$ in some interval $(R,R+\epsilon)$, the number $R$ cannot be improved.If the functions $\varphi_n(x)\;( n\geq0)$ are smooth functions then the last condition is equivalent to the inequality $	\varphi^{\prime}_0(R)<\frac{2}{p}\sum_{n=1}^{\infty}\varphi^{\prime}_n(R).$
\end{thm}
Using Theorem \ref{Th-3.2}, the authors found the value of the Bohr radius of the identity operator and the convolution operator with hypergeometric Gasussian function. Let 
\begin{align*}
	F(z)=F(a,b,c,z)=\sum_{n=0}^{\infty}\gamma_nz^n,\; \gamma_n:=\frac{(a)_n(b)_n}{(c)_n(1)_n},
\end{align*}
where $a, b, c>-1$, such that $\gamma_n\geq 0$, $(a)_n:=a(a+1)\cdots(a+n-1)$, $(a)_0=1$. Then  for $p\in (0, 2]$
\begin{align*}
	|a_0|^p+\sum_{n=1}^{\infty}|\gamma_n||a_n|r^n\leq \; \mbox{for all}\; r\leq R,
\end{align*}
where $R$ is the minimal positive root of the equation $|F(a,b,c,x)-1|=p/2$. Recently, Chen \emph{et al.} \cite{Chen-Liu-Pon-RM-2023} established several new versions of Bohr-type inequalities for bounded analytic functions in the unit disk $\mathbb{D}$ by allowing $\{\varphi_n(r)\}^{\infty}_{n=0}$ in the place of $\{r^n\}^{\infty}_{n=0}$ in the power series representations of the functions involved with the Bohr sum and thereby introducing a single parameter, which generalized several related results.\vspace{1.2mm}

Inspired by the idea of the proof technique of   \cite{Kayumov-MJM-2022,Chen-Liu-Pon-RM-2023}, our aim is to establish refined Bohr inequality, finding the corresponding sharp radius for the class $ \mathcal{P}^{0}_{\mathcal{H}}(M) $ which has been studied by Ghosh and Vasudevarao (see \cite{Ghosh-Vasudevarao-BAMS-2020}) 
\begin{align*}
	\mathcal{P}^{0}_{\mathcal{H}}(M)=\{f=h+\overline{g} \in \mathcal{H}_{0}: \real (zh^{\prime\prime}(z))> -M+|zg^{\prime\prime}(z)|, \; z \in \mathbb{D}\; \mbox{and }\; M>0\}.
\end{align*}
To study Bohr inequality and Bohr radius for functions in $ \mathcal{P}^{0}_{\mathcal{H}}(M) $, we require the coefficient bounds and growth estimate of functions in $ \mathcal{P}^{0}_{\mathcal{H}}(M) $. We have the following result on the coefficient bounds and growth estimate for functions in $ \mathcal{P}^{0}_{\mathcal{H}}(M) $.
\begin{lem} \label{lem-7.1} \cite{Ghosh-Vasudevarao-BAMS-2020}
	Let $f=h+\overline{g}\in \mathcal{P}^{0}_{\mathcal{H}}(M)$ be given by \eqref{e-7.2} for some $M>0$. Then for $n\geq 2,$ 
	\begin{enumerate}
		\item[(i)] $\displaystyle |a_n| + |b_n|\leq \frac {2M}{n(n-1)}; $\\[2mm]
		
		\item[(ii)] $\displaystyle ||a_n| - |b_n||\leq \frac {2M}{n(n-1)};$\\[2mm]
		
		\item[(iii)] $\displaystyle |a_n|\leq \frac {2M}{n(n-1)}.$
	\end{enumerate}
	The inequalities  are sharp with extremal function   $f_M$ given by 
	$f_M^{\prime}(z)=1-2M\, \ln\, (1-z) .$	
\end{lem}
\begin{lem}\cite{Ghosh-Vasudevarao-BAMS-2020}\label{lem-7.2}
	Let $f \in \mathcal{P}^{0}_{\mathcal{H}}(M)$ be given by \eqref{e-7.2}. Then 
	\begin{equation*}
		|z| +2M \sum\limits_{n=2}^{\infty} \dfrac{(-1)^{n-1}|z|^{n}}{n(n-1)} \leq |f(z)| \leq |z| + 2M \sum\limits_{n=2}^{\infty} \dfrac{|z|^{n}}{n(n-1)}.
	\end{equation*}
	Both  inequalities are sharp for the function $f_{M}$ given by $f_{M}(z)=z+ 2M \sum\limits_{n=2}^{\infty} \dfrac{z^n}{n(n-1)}.
	$
\end{lem}
For the class $ \mathcal{P}^{0}_{\mathcal{H}}(M) $, Allu and Halder (see \cite{Allu-BSM-2021}) studied the Bohr inequality in terms of distance formulations and obtained the following sharp Bohr inequality.
\begin{thm}\label{Th-3.4}\cite[Theorem 2.9]{Allu-BSM-2021}
Let $f\in \mathcal{P}^{0}_{\mathcal{H}}(M) $ be given by \eqref{e-7.2} with  $0<M<1/(2(\ln 4-1))$ . Then
\begin{align*}
|z|+\sum_{n=2}^{\infty}\left(|a_n|+|b_n|\right)|z|^n\leq d\left(f(0), \partial f(\mathbb{D})\right)
\end{align*} holds for $|z|=r\leq r_f$, where $r_f$ is the unique root of 
\begin{align*}
r+2M\sum_{n=2}^{\infty}\frac{r^n}{n(n-1)}=1+2M\sum_{n=2}^{\infty}\frac{(-1)^{n-1}}{n(n-1)}
\end{align*} in $(0,1)$. The radius $r_f$ is best possible.
\end{thm}
Let $\mathcal{F}$ denote the set of all sequences $\{\varphi_n(r)\}^{\infty}_n$ of non negative continuous functions in $[0,1)$ such that the series $\sum_{n=0}^{\infty}\varphi_n(r)$ converges locally uniformly on the interval $[0,1)$ $\;\forall\; n\in\mathbb{N}\cup \{0\}$.  It is natural to raise the following question for further exploration of Bohr inequality for the class  $\mathcal{P}^{0}_{\mathcal{H}}(M)$.
\begin{ques}\label{Qn-3.1}
In view of the general setting proposed in \cite{Kayumov-MJM-2022} with a change of basis from $\{r^n\}_{n=0}^\infty$ to $\{\varphi_n(r)\}_{n=0}^{\infty}$, is it possible to establish a general result or improve Theorem \ref{Th-3.4} in the setting of one parameter family of Bohr sum?
\end{ques}
We begin by presenting our result concerning the Bohr inequality for the class $\mathcal{P}^{0}_{\mathcal{H}}(M)$ in terms of distance formulation, which is formulated in response to Question \ref{Qn-3.1}. It is worth noticing that, if, in particular, $\varphi_n(r)=r^{n}$, $n\in\mathbb{N}\cup\{0\}$ in Theorem \ref{th-7.3}, then Theorem \ref{Th-3.4} becomes a special case of Theorem \ref{th-7.3}.
\begin{thm}\label{th-7.3}
	Let $f=h+\overline{g}\in \mathcal{P}^{0}_{\mathcal{H}}(M)$ be given by \eqref{e-7.2} and suppose that $\{\varphi_n(r)\}^{\infty}_{n=0}\in\mathcal{F}$ with condition $\varphi_0(1)=1$ and  
	\begin{align}\label{ee-7.6}
		\sum_{n=2}^{\infty}\frac{\varphi_n(0)}{n(n-1)}<\frac{1}{2M}+1-\ln 4.
	\end{align}
Then
\begin{align*}
A_f(r)=r\varphi_0(r)+\sum_{n=2}^{\infty}\left(|a_n|+|b_n|\right)\varphi_n(r)\leq d\left(f(0), \partial f(\mathbb{D})\right),
\end{align*}for $|z|=r\leq R_f(M),$ where $R_f(M)$ is the unique positive root in $(0,1)$ of 
\begin{align*}
r\varphi_0(r)+2M\sum_{n=2}^{\infty}\frac{\varphi_n(r)}{n(n-1)}=1+2M\sum_{n=2}^{\infty}\frac{(-1)^{n-1}}{n(n-1)}.
\end{align*}The radius $R_f(M)$ is best possible.
\end{thm}
\subsection{Application of Theorem \ref{th-7.3}}
Let $f=h+\overline{g}\in \mathcal{P}^{0}_{\mathcal{H}}(M)$ be given by \eqref{e-7.2} for  $0<M<1/2(\ln 4-1)$. A simple computation shows that
\begin{align*}
	\begin{cases}
		\displaystyle\sum_{n=2}^{\infty}\frac{nr^n}{n(n-1)}=-r\ln(1-r),\vspace{2mm}\\	\displaystyle\sum_{n=2}^{\infty}\frac{n^2r^n}{n(n-1)}=\frac{r[r-(1-r)\ln(1-r)]}{1-r},\vspace{2mm}\\
		\displaystyle\sum_{n=2}^{\infty}\frac{n^3r^n}{n(n-1)}=r\left(\frac{(3-2r)r}{(1-r)^2}-\ln (1-r)\right).
	\end{cases}
\end{align*}
\begin{enumerate}
	\item[(i)] For $ f=h+\overline{g}\in \mathcal{P}^{0}_{\mathcal{H}}(M)$ and for $\varphi_0(r)=1$ and $\varphi_n(r)=n^{\alpha}r^n\;(n\geq1)$, we define the following functional
	\begin{align*}
		\mathcal{C}^f_{\alpha}(r):=r\varphi_0(r)+\sum_{n=2}^{\infty}\left(|a_n|+|b_n|\right)\varphi_n(r).
	\end{align*} 
	We see that $\varphi_0(1)=1$ and the condition
	\begin{align*}
		\sum_{n=2}^{\infty}\frac{\varphi_n(0)}{n(n-1)}<\frac{1}{2M}+1-\ln 4\; \mbox{coincides with}\; 0<M<\frac{1}{2(\ln 4-1)}.
	\end{align*}
	Then, we have
	\begin{enumerate}
		\item[(a)] $\mathcal{C}^f_{1}(r)\leq d\left(f(0), \partial f(\mathbb{D})\right)$ for $|z|=r\leq R_1(M)$, where $R_1(M)$ is the unique positive root in $(0,1)$ of the equation $r-2Mr\ln(1-r)=1+2M(1-\ln 4)$. The number $R_1(M)$ is best possible.
		\item[(b)] $\mathcal{C}^f_{2}(r)\leq d\left(f(0), \partial f(\mathbb{D})\right)$ 
		for $|z|=r\leq R_2(M)$, where $R_2(M)$ is the unique positive root in $(0,1)$ of the equation 
		\begin{align*}
			r+\frac{2Mr[r-(1-r)\ln(1-r)]}{1-r}=1+2M(1-\ln 4).
		\end{align*}
		The number $R_2(M)$ is best possible.
		\item[(c)] $\mathcal{C}^f_{3}(r)\leq d\left(f(0), \partial f(\mathbb{D})\right)$ for $|z|=r\leq R_3(M)$, where $R_3(M)$ is the unique positive root in $(0,1)$ of the equation 
		\begin{align*}
			r+2Mr\left(\frac{(3-2r)r}{(1-r)^2}-\ln (1-r)\right)=1+2M(1-\ln 4).
		\end{align*}
		The number $R_3(M)$ is best possible.
	\end{enumerate}
	\begin{table}[ht]
		\centering
		\begin{tabular}{|l|l|l|l|l|l|l|l|l|l|}
			\hline
			$M$& $0.431 $&$0.862 $& $1.210 $& $1.271$& $1.289 $&$1.292 $&$ 1.2935 $ &$1.29421$& $1.29433$ \\
			\hline
			$R_{1}(M)$& $0.443$&$0.230 $& $0.057$& $0.017$& $0.0040 $&$0.0018 $& $0.00065$& $0.00010$ &$ 0.000015 $\\
			\hline
			$R_{2}(M)$& $0.358$&$0.189 $& $0.029$& $0.016$& $0.0040 $&$0.0017 $& $0.00065$& $0.00010$ &$ 0.000015 $\\
			\hline
			$R_{3}(M)$& $0.277$&$0.149 $& $0.044$& $0.015$& $0.0039 $&$0.0017 $& $0.00065$& $0.00010$ &$ 0.000015 $\\
			\hline
		\end{tabular}\vspace{2.5mm}
		\caption{This table exhibits the the approximate values of the roots $ R_{\alpha}(M) $ for different values of $ M$ and $\alpha=1,2,3$, where $0<M<1/(2(\ln4-1))\approx 1.29435$. Moreover, we see that $ \lim\limits_{M\rightarrow 1.29435}R_{\alpha}(M)\approx 0.000015$.}
	\end{table}
	\item[(ii)]  A simple computation shows that
	\begin{align*}
		\begin{cases}
			\displaystyle\sum_{n=2}^{\infty}\frac{(n+1)r^n}{n(n-1)}=r+(1-2r)\ln(1-r),\vspace{2mm}\\	\displaystyle\sum_{n=2}^{\infty}\frac{(n+1)^2r^n}{n(n-1)}=\frac{r+(1-5r+4r^2)\ln(1-r)}{1-r},\vspace{2mm}\\
			\displaystyle\sum_{n=2}^{\infty}\frac{(n+1)^3r^n}{n(n-1)}=\frac{r+4r^2-4r^3+(1-r)^2(1-8r)\ln (1-r)}{(1-r)^2}.
		\end{cases}
	\end{align*} 
	For $ f=h+\overline{g}\in \mathcal{P}^{0}_{\mathcal{H}}(M)$ and for $\varphi_n(r)=(n+1)^{\beta}r^n\;(n\geq1)$, we define the following functional
	\begin{align*}
		\mathcal{D}^f_{\beta}(r):=r\varphi_0(r)+\sum_{n=2}^{\infty}\left(|a_n|+|b_n|\right)\varphi_n(r).
	\end{align*}
	We see that $\varphi_0(1)=1$ and the condition
	\begin{align*}
		\sum_{n=2}^{\infty}\frac{\varphi_n(0)}{n(n-1)}<\frac{1}{2M}+1-\ln 4\; \mbox{coincides with}\; 0<M<\frac{1}{2(\ln 4-1)}.
	\end{align*}
	Then, we have
	\begin{enumerate}
		\item[(a)] $\mathcal{D}^f_{1}(r) \leq d\left(f(0), \partial f(\mathbb{D})\right)$ for $|z|=r\leq R^*_1(M)$, where $R^*_1(M)$ is the unique positive root in $(0,1)$ of the equation $r+2M[r+(1-2r)\ln(1-r)]=1+2M(1-\ln 4).$ The number $R_1^*(M)$ is best possible.
		\item[(b)] $\mathcal{D}^f_{2}(r) \leq d\left(f(0), \partial f(\mathbb{D})\right)$ for $|z|=r\leq R^*_2(M)$, where $R_2(M)$ is the unique positive root in $(0,1)$ of the equation 
		\begin{align*}
			r+\frac{2M[r+(1-5r+4r^2)\ln(1-r)]}{1-r}=1+2M(1-\ln 4).
		\end{align*}
		The number $R^*_2(M)$ is best possible.
		\item[(c)] $\mathcal{D}^f_{3}(r) \leq d\left(f(0), \partial f(\mathbb{D})\right)$
		for $|z|=r\leq R^*_3(M)$, where $R^*_3(M)$ is the unique positive root in $(0,1)$ of the equation 
		\begin{align*}
			r+\frac{2M[r+4r^2-4r^3+(1-r)^2(1-8r)\ln (1-r)]}{(1-r)^2}=1+2M(1-\ln 4).
		\end{align*}
		The number $R^*_3(M)$ is best possible.
	\end{enumerate}
\end{enumerate}
\begin{table}[ht]
	\centering
	\begin{tabular}{|l|l|l|l|l|l|l|l|l|l|}
		\hline
		$M$& $0.431 $&$0.862 $& $1.210 $& $1.271$& $1.289 $&$1.292 $&$ 1.2935 $ &$1.29421$& $1.29433$ \\
		\hline
		$R^*_{1}(M)$& $0.404$&$0.208 $& $0.054$& $0.016$& $0.0040 $&$0.0018 $& $0.00065$& $0.00010$ &$ 0.000015 $\\
		\hline
		$R^*_{2}(M)$& $0.284$&$0.147 $& $0.043$& $0.015$& $0.0039 $&$0.0017 $& $0.00065$& $0.00010$ &$ 0.000015 $\\
		\hline
		$R^*_{3}(M)$& $0.203$&$0.147 $& $0.043$&  $0.015$& $0.0039 $&$0.0017 $& $0.00065$& $0.00010$ &$ 0.000015 $\\
		\hline
	\end{tabular}\vspace{2.5mm}
	\caption{This table exhibits the the approximate values of the roots $ R^*_{\beta}(M) $ for different values of $ M$ and  $\beta=1,2,3$, where $0<M<1/(2(\ln 4-1))\approx 1.29435$. Moreover, we see that $ \lim\limits_{M\rightarrow 1.29435}R^*_{\beta}(M)\approx 0.000015$.}
\end{table}
We now discuss the proof of Theorem \ref{th-7.3} in details.
\subsection{\bf Proof of Theorem \ref{th-7.3}} 
Let $f=h+\overline{g}\in \mathcal{P}^{0}_{\mathcal{H}}(M)$ be given by \eqref{e-7.2} for  $0<M<1/2(\ln 4-1)$. By the Lemma \ref{lem-7.2}, it is evident that the Euclidean distance between $f(0)$ and the boundary of $ f(\mathbb{D})$ is
\begin{align}\label{e-7.6}
   d\left(f(0), \partial f(\mathbb{D})\right)\geq 1+2M\sum_{n=2}^{\infty}\frac{(-1)^{n-1}}{n(n-1)}.
\end{align}
Let $H_M : [0, 1]\to\mathbb{R}$ be defined by 
\begin{align}\label{e-7.7}
	H_M(r)=r\varphi_0(r)+2M\sum_{n=2}^{\infty}\frac{\varphi_n(r)}{n(n-1)}-1-2M\sum_{n=2}^{\infty}\frac{(-1)^{n-1}}{n(n-1)}.
\end{align}
Clearly, $H_M$ is continuous in $[0,1]$ and differentiable in $(0,1)$. Moreover, in view of the condition in \eqref{ee-7.6}, we note that
  \begin{align*}
  H_M(0)=2M\sum_{n=2}^{\infty}\frac{\varphi_n(0)}{n(n-1)}-1-2M\sum_{n=2}^{\infty}\frac{(-1)^{n-1}}{n(n-1)}<0.
  \end{align*}
  Since $\varphi_0(1)=1$ and $1-\ln 4\approx -0.3862<0$, we see that
  \begin{align*}
  H_M(1)=&\varphi_0(1)+2M\sum_{n=2}^{\infty}\frac{\varphi_n(1)}{n(n-1)}-1-2M\sum_{n=2}^{\infty}\frac{(-1)^{n-1}}{n(n-1)}\\&=2M\left(\sum_{n=2}^{\infty}\frac{\varphi_n(1)}{n(n-1)} -(1-\ln 4)\right)>0.
  \end{align*}
  By the \textit{Intermediate Value Theorem}, the function $H_M$ has a root in $(0,1)$, $R_f(M)$ say. To show that $R_f(M)$ is unique, it is sufficient to show that $H_M$ is a monotonic function on $[0, 1]$. An easy computation shows that
  \begin{align*}
  H_M^{\prime}(r)=r\varphi^{\prime}_0(r)+\varphi_0(r)+2M\sum_{n=2}^{\infty}\frac{\varphi^{\prime}_n(r)}{n(n-1)}>0\; \mbox{for all}\; r\in (0,1).
  \end{align*} 
Thus, we conclude that $H_M$ is a strictly monotone increasing function on $(0,1)$. However, we see that
  \begin{align}\label{e-7.8}
  R_f(M)\varphi_0(R_f(M))+2M\sum_{n=2}^{\infty}\frac{\varphi_n(R_f(M))}{n(n-1)}=1+2M\sum_{n=2}^{\infty}\frac{(-1)^{n-1}}{n(n-1)}
  \end{align} 
  The Lemma \ref{lem-7.1} in view of \eqref{e-7.6} yields that 
  \begin{align*}
  	A_f(r)&=r\varphi_0(r)+\sum_{n=2}^{\infty}\left(|a_n|+|b_n|\right)\varphi_n(r)\\&\leq r\varphi_0(r)+2M\sum_{n=2}^{\infty}\frac{\varphi_n(r)}{n(n-1)}\\&\leq 1+2M\sum_{n=2}^{\infty}\frac{(-1)^{n-1}}{n(n-1)}\\&\leq d\left(f(0), \partial f(\mathbb{D})\right) 
  \end{align*} 
  for $|z|=r\leq R_f(M)$. Thus the desired inequality is established.\vspace{1.2mm}
  
  The next part of the proof is to show that the number $R_f(M)$ is best possible. Henceforth, we consider the function $f_M$ given by
  \begin{align*}
  	f_M(z)=z+2M\sum_{n=2}^{\infty}\frac{z^n}{n(n-1)}.
 \end{align*}
 Clearly, we see that $f_M\in \in \mathcal{P}^{0}_{\mathcal{H}}(M)$. For $r>R_f(M)$ and $f=f_M$, an easy computation in view of part-(i) of Lemma \ref{lem-7.1} and \eqref{e-7.8} shows that
 \begin{align*}
 	A_{f_M}(r)&=r\varphi_0(r)+\sum_{n=2}^{\infty}\left(|a_n|+|b_n|\right)\varphi_n(r)\\&=r\varphi_0(r)+2M\sum_{n=2}^{\infty}\frac{\varphi_n(r)}{n(n-1)}\\&>R_f(M)\varphi_0(R_f(M))+2M\sum_{n=2}^{\infty}\frac{\varphi_n(R_f(M))}{n(n-1)}\\&=1+2M\sum_{n=2}^{\infty}\frac{(-1)^{n-1}}{n(n-1)}\\&=d\left(f_M(0), \partial f_M(\mathbb{D})\right).
 \end{align*}
This turns out that the number $R_f(M)$ is best possible.

%\noindent{\bf Acknowledgment:} The authors would greatly indebted to the anonymous referees for their elaborate comments and valuable suggestions which improve significantly the presentation of the paper. The second author is supported by UGC-JRF (NTA Ref. No.: $211610135410$), New Delhi, India. 
\vspace{3mm}

\noindent\textbf{Compliance of Ethical Standards}\\

\noindent\textbf{Conflict of interest} The authors declare that there is no conflict  of interest regarding the publication of this paper.\vspace{1.5mm}

\noindent\textbf{Data availability statement}  Data sharing not applicable to this article as no datasets were generated or analyzed during the current study.

\end{document}